\documentclass[a4paper,12pt]{amsart}
\usepackage{a4wide}
\usepackage{eucal}
\usepackage{enumerate}
\usepackage{easymat}
\usepackage{multirow}
\usepackage{color}
\usepackage[colorlinks=true,linkcolor=blue]{hyperref}
\usepackage{bbm}
\usepackage[utf8]{inputenc}
\usepackage[T1]{fontenc}

\numberwithin{equation}{section}

\advance\textheight by 2cm

\theoremstyle{plain}
\newtheorem{thm}{Theorem}[section]
\newtheorem{cor}[thm]{Corollary}
\newtheorem{lem}[thm]{Lemma}

\theoremstyle{definition}
\newtheorem{defn}[thm]{Definition}

\theoremstyle{remark}
\newtheorem{rem}[thm]{Remark}
\newtheorem{rems}[thm]{Remarks}

\newcounter{stp}

\newcommand{\RR}{\mathbb{R}}
\newcommand{\CC}{\mathbb{C}}

\newcommand{\NN}{\mathbb{N}}

\newcommand{\sL}{\mathcal{L}}

\newcommand{\sLX}{{\sL(X)}}

\newcommand{\sB}{\mathcal{B}}
\newcommand{\sC}{\mathcal{C}}

\newcommand{\sF}{\mathcal{F}}
\newcommand{\sR}{\mathcal{R}}

\newcommand{\Fi}{\mathcal{F}_\infty}
\newcommand{\Fim}{\mathcal{F}^\mu_\infty}
\newcommand{\sX}{\mathcal{X}}
\newcommand{\sx}{{\scriptstyle\mathcal{X}}}

\newcommand{\sG}{\mathcal{G}}

\newcommand{\sT}{\mathcal{T}}

\newcommand{\rL}{\mathrm{L}}
\newcommand{\rC}{\mathrm{C}}
\newcommand{\rW}{\mathrm{W}}
\newcommand{\p}{{\raisebox{1.3pt}{{$\scriptscriptstyle\bullet$}}}}
\newcommand{\ps}{{\raisebox{0.3pt}{{$\scriptscriptstyle\bullet$}}}}
\newcommand{\U}{U}
\newcommand{\Y}{U}
\newcommand{\LpntnU}{{\rL^p\bigl([0,t_0],\U\bigr)}}
\newcommand{\LpntnY}{{\rL^p\bigl([0,t_0],\Y\bigr)}}
\newcommand{\LpniU}{{\rL^p\bigl([0,+\infty),\U\bigr)}}
\newcommand{\LpniY}{{\rL^p\bigl([0,+\infty),\Y\bigr)}}

\newcommand{\LpntU}{{\rL^p\bigl([0,t],\U\bigr)}}

\newcommand{\LantnU}{{\rL^\a\bigl([0,\tn],\U\bigr)}}
\newcommand{\LbntnU}{{\rL^\b\bigl([0,\tn],\U\bigr)}}
\newcommand{\LsntU}{{\rL^s\bigl([0,t],\U\bigr)}}
\newcommand{\LrntU}{{\rL^r\bigl([0,t],\U\bigr)}}

\newcommand{\LpntnX}{{\rL^p\bigl([0,t_0],X\bigr)}}

\newcommand{\CntnX}{{\rC\bigl([0,t_0],X\bigr)}}
\newcommand{\Cne}{{\rC[0,1]}}

\newcommand{\LantU}{{\rL^\a\bigl([0,t],\U\bigr)}}
\newcommand{\LbntU}{{\rL^\b\bigl([0,t],\U\bigr)}}
\newcommand{\LbntnX}{{\rL^\b\bigl([0,\tn],X\bigr)}}
\newcommand{\LentnX}{{\rL^1\bigl([0,\tn],X\bigr)}}

\newcommand{\WnzantnU}{\rW^{2,\a}_0\bigl([0,\tn],\U\bigr)}
\newcommand{\WnzpntnU}{\rW^{2,p}_0\bigl([0,\tn],\U\bigr)}
\newcommand{\WzpntnU}{\rW^{2,p}\bigl([0,\tn],\U\bigr)}

\newcommand{\WnzantU}{\rW^{2,\a}_0\bigl([0,t],\U\bigr)}

\newcommand{\WpntnU}{{\rW^{1,p}\bigl([0,t_0],\U\bigr)}}

\newcommand{\Lpne}{\rL^p[0,1]}
\newcommand{\Lpntn}{\rL^p[0,\tn]}
\newcommand{\Wepne}{\rW^{1,p}[0,1]}

\newcommand{\Wepnntn}{\rW^{1,p}_0[0,\tn]}

\newcommand{\LpntUs}{{\rL^p([0,t],\U)}}

\newcommand{\WnzpniU}{\rW^{2,p}_0\bigl([0,+\infty),\U\bigr)}

\newcommand{\LpnntnU}{{\rL^p\bigl([0,n\tn],\U\bigr)}}

\newcommand{\RlA}{R(\lambda,A)}
\newcommand{\RlAme}{R(\lambda,\Ame)}

\newcommand{\Tt}{(T(t))_{t\ge0}}
\newcommand{\St}{(S(t))_{t\ge0}}

\newcommand{\D}{D}
\newcommand{\ds}{\,ds}
\newcommand{\dmr}{\,d\mu(r)}
\newcommand{\dms}{\,d\mu(s)}

\newcommand{\dmas}{\,d|\mu|(s)}
\newcommand{\dmar}{\,d|\mu|(r)}

\newcommand{\dr}{\,dr}
\newcommand{\gb}{\operatorname{\omega_0}}

\newcommand{\rg}{\operatorname{rg}}
\newcommand{\supp}{\operatorname{supp}}
\renewcommand{\Re}{\operatorname{Re}}
\newcommand{\inc}{\overset{\text{c}\;}{\hookrightarrow}}
\newcommand{\Xme}{X_{-1}^A}
\newcommand{\Ame}{A_{-1}}
\newcommand{\Tme}{T_{-1}}
\newcommand{\tn}{{t_0}}
\newcommand{\Vt}{{F_t}}
\newcommand{\Vtn}{{F_\tn}}
\newcommand{\Bt}{\sB_t}
\newcommand{\Btt}{(\Bt)_{t\ge0}}

\newcommand{\Btn}{\sB_{\tn}}
\newcommand{\Ctn}{\sC_{\tn}}
\newcommand{\Ci}{\sC_{\infty}}

\newcommand{\Ft}{\mathcal{F}_t}
\newcommand{\Ftn}{\mathcal{F}_\tn}
\newcommand{\Ftnm}{\Ftn^\mu}

\newcommand{\Fte}{\sF_{t_1}}
\newcommand{\Fntn}{\sF_{n\tn}}
\newcommand{\Fntnm}{\Fntn^\mu}
\newcommand{\Fmtnm}{\sF_{m\tn}^\mu}

\newcommand{\Ql}{Q(\lambda)}

\newcommand{\Ji}{J^{-1}}

\renewcommand{\a}{\alpha}
\renewcommand{\b}{\beta}
\newcommand{\aaa}{\gamma}
\newcommand{\bb}{\delta}
\newcommand{\eins}{\mathbbm{1}}
\newcommand{\eab}{\eins_{[\aaa,\bb]}}
\newcommand{\eabk}{\eins_{[\aaa_k,\bb_k]}}
\newcommand{\ABC}{A_{BC}}
\newcommand{\RlABC}{R(\lambda,\ABC)}

\renewcommand{\r}{\right}
\renewcommand{\l}{\left}

\newcommand{\Mem}{M_{\varepsilon_\mu}}
\newcommand{\Mmem}{M_{\varepsilon_{-\mu}}}

\newcommand{\kommentar}[1]{\bgroup \egroup}

\begin{document}
\title[Perturbation of generators]
{On perturbations of generators\\ of \boldmath{$C_0$}-semigroups}
\author{Martin Adler, Miriam Bombieri and Klaus-Jochen Engel}

\address{Martin Adler\\
Arbeitsbereich Funktionalanalysis \\
Mathematisches Institut\\
Auf der Morgenstelle 10 \\
D-72076 T\"{u}bingen}
\email{maad@fa.uni-tuebingen.de}

\address{Miriam Bombieri\\
Arbeitsbereich Funktionalanalysis \\
Mathematisches Institut\\
Auf der Morgenstelle 10 \\
D-72076 T\"{u}bingen}
\email{mibo@fa.uni-tuebingen.de}

\address{Klaus-Jochen Engel\\
Universit\`{a} degli Studi dell'Aquila \\
Dipartimento di Ingegneria e Scienze dell'Informazione e Matematica (DISIM)\\
Via Vetoio\\
I-67100 L'Aquila (AQ)}
\email{klaus.engel@univaq.it}

\begin{abstract}
We present a perturbation result for generators of $C_0$-semigroups which can be considered as an operator theoretic version of the Weiss--Staffans perturbation theorem for abstract linear systems. The results are illustrated by applications to the Desch--Schappacher, the Miyadera--Voigt perturbation theorems, and to unbounded perturbations of the boundary conditions of a generator.
\end{abstract}
\date{\today}%
\maketitle

\section{Introduction}

In his classic \cite{Kat:66} \emph{``Perturbation Theory for Linear Operators''}, Tosio Kato addresses, among others, the following general problem:

\medbreak
\emph{Given (unbounded) operators $A$ and $P$ on a Banach space $X$, how
should one define their ``sum'' $A + P$ and which properties of $A$ are preserved under the perturbation by $P$?}
\medbreak

In the present paper we study this problem in the context of operator semigroups. Given the generator $A$ of a $C_0$-semigroup on $X$, for which operators $P$  is the (in a suitable way defined) sum $A + P$ again a generator?

\smallbreak
Numerous results are known in this direction (see, e.g., \cite[Sects.~III.1--3 \& relative Notes]{EN:00}), but no unifying and general theory is yet available.

\smallbreak
Our aim is to go a step towards a more systematic perturbation
theory for such generators. To this end we choose the
following setting.

\smallbreak
For the generator $A$ with domain $D(A)\subset X$ we consider perturbations
\[P : D(P )\subset X \to X_{-1}^A,\]
where $X_{-1}^A$ is the extrapolated space associated to $A$ (see \cite[Sect.~II.5.a]{EN:00}). The sum is then defined as $A_P := (A_{-1} + P )|_{X}$, i.e.,
\[
A_P x = A_{-1} x + P x
\quad\text{for }x \in D(A_P):=\bigl\{z \in D(P ) : A_{-1} z + P z \in X\bigr\} .
\]
We then ask: For which $P$ remains $A_P$ a generator on $X$?
The bounded perturbation theorem (\cite[Sect.~III.1]{EN:00}), the Desch--Schappacher (\cite[Sect.~III.3.a]{EN:00}) and the Miyadera--Voigt theorems (\cite[Sect.~III.3.c]{EN:00}) give some well-known specific answers
to this question.

\smallbreak
Our starting point is the Weiss--Staffans theorem from control theory on the well-posedness of perturbed linear systems,  cf. \cite[Sects.~7.1 \& 7.4]{Sta:05}. We formulate and prove this result in a purely operator theoretic way avoiding, in particular, notions like abstract linear system and Lebesgue- or Yosida extensions.

\smallbreak
Before presenting our approach, it might still be helpful to give first some background from control theory.

\smallbreak
One view at our approach is to interpret the perturbed generator as the system operator of a control system with feedback.
More precisely, we take two Banach spaces $X$ and $U$ called \emph{state-} and \emph{observation-}/\emph{control space}\footnote{In the language of control theory we assume that the observation and control spaces coincide which, in our context, is no restriction of generality and somewhat simplifies the presentation.}, respectively.
On these spaces we consider the operators
\begin{itemize}
	\item $A:D(A)\subset X \to X$, called the \emph{state operator} (of the unperturbed system),
	\item $B\in\sL(U,\Xme)$, called the \emph{control operator},
	\item $C\in\sL(Z,\Y)$, called the \emph{observation operator},
\end{itemize}
where we assume that $A$ is the generator of a $C_0$-semigroup $\Tt$ on $X$. Moreover, $D(C)=Z$ is a Banach space such that
\[X_1^A\inc Z\inc X,\]
where by ``$\inc$'' we denote a continuous injection and $X_1^A$ is the domain $D(A)$ equipped with the graph norm.
We then consider the linear control system
\begin{equation*}
\tag*{$\Sigma(A,B,C)$}
\label{csu}
\begin{cases}
\dot{x}(t)=Ax(t)+Bu(t),& t\geq0, \\
y(t)=Cx(t),& t\geq0, \\
x(0)=x_0.
\end{cases}
\end{equation*}
The solution of \ref{csu} is formally  given by the variation of parameters formula
\begin{equation}
\label{eq:vpf}
x(t)=T(t)x_0+\int_0^t\Tme(t-s)Bu(s)\ds.
\end{equation}
If we close this system by putting $u(t)=y(t)$, we formally obtain the perturbed abstract Cauchy problem
\begin{equation}
\label{eq:acp-BC}
\begin{cases}
\dot{x}(t)=(\Ame+BC)x(t),& t\geq0, \\
x(0)=x_0,
\end{cases}
\end{equation}
which is well-posed in $X$ if and only if $A_P$ for $P:=BC\in\sL(Z,\Xme)$ is a generator on $X$, cf. \cite[Sect.~II.6]{EN:00}.

\medbreak
Before elaborating more on this idea we give a short summary of our paper.

\smallbreak
Section~\ref{sec:admiss} is dedicated to the notions of \emph{admissibility} for control-, observation- and pairs of operators. In Section~\ref{sec:SW} we state and prove our main results, i.e., Theorems~\ref{thm:WS} and \ref{thm-main}. In Section~\ref{sec:app} we show how the Desch--Schappacher and Miyadera--Voigt theorems easily follow from Theorem~\ref{thm-main}. Moreover, we give an application to the perturbation of the boundary condition of a generator in the spirit of Greiner \cite{Gre:87}.  Finally, in the appendix we give an estimate for the norm of a Toeplitz block-operator matrix which is needed to prove our main result.


\section{Admissibility}\label{sec:admiss}

If in the system \ref{csu} we take $C=0$ and consider the initial value $x_0=0$, then it is natural to ask that for every control function $u\in\LpntnU$ we obtain a state $x(\tn)\in X$ for some/all $\tn>0$. Hence by formula~\eqref{eq:vpf} we are led to the following definition, cf. \cite[Def.~4.1]{Wei:89a}, see also \cite{Eng:98}.

\begin{defn} The control operator $B\in\sL(U,\Xme)$ is called \emph{$p$-admissible} for some $1\le p<+\infty$ if there exists $\tn>0$ such that
\begin{equation}\label{vor-1U}
\int_0^{t_0} T_{-1}(t_0-s) B u(s)\ds \in X\quad\text{for all }u\in\LpntnU.
\end{equation}
\end{defn}

Note that \eqref{vor-1U} becomes less restrictive for growing $p\in[1,+\infty)$.


\begin{rem}\label{rem:Btn}
The range condition~\eqref{vor-1U} in the previous definition means that the operator $\Btn:\LpntnU\to\Xme$ given by
\begin{equation}\label{def:Bt}
\Btn u:=\int_0^{t_0} T_{-1}(t_0-s) B u(s)\ds,\quad u\in\LpntnU
\end{equation}
has range $\rg(\Btn)\subseteq X$. Since obviously $\Btn\in\sL\bigl(\LpntnU,\Xme\bigr)$, the closed graph theorem implies that for admissible $B$ the \emph{controllability map} $\Btn$ belongs to $\sL\bigl(\LpntnU,X\bigr)$. On the other hand, using integration by parts, it follows that for every $u\in\WpntnU$
\begin{align*}
\int_0^{t_0}T_{-1}(t_0-s) B u(s)\ds
&=\Ame^{-1}\biggl(T_{-1}(t_0)Bu(0)-Bu(t_0) + \int_0^{t_0}T_{-1}(t_0-s) B u' (s)\ds\biggr)\\
&\in X.
\end{align*}
Since $\WpntnU$ is dense in $\LpntnU$, this shows that the range condition~\eqref{vor-1U} is equivalent to the existence of some $M\ge0$ such that
\begin{equation}\label{eq:add-B-M}
\biggl\|\int_0^{t_0} T_{-1}(t_0-s) B u(s)\ds\biggr\|_X\le M\cdot\|u\|_p
\quad\text{for all }u\in\WpntnU.
\end{equation}
\end{rem}

\bigbreak
Next we consider \ref{csu} with $B=0$. Then it is reasonable to ask that every initial value $x_0\in D(A)$ gives rise to an observation $y(\p)=CT(\p)x_0\in\LpntnY$ for some/all $\tn>0$ which also depends continuously on $x_0$.  This yields the following definition, cf. \cite[Def.~6.1]{Wei:89b}, see also \cite{Eng:98}.

\begin{defn} The observation operator $C\in\sL(Z,\Y)$ is called \emph{$p$-admissible} for some $1\le p<+\infty$ if there exist $t_0>0$ and $M\ge0$ such that
\begin{equation}\label{vor-2U}
\int_0^{t_0}\bigl\|CT(s)x\bigr\|_\Y^p\ds \leq M\cdot\|x\|_X^p
\quad\text{for all }x\in D(A).
\end{equation}
\end{defn}

Note that \eqref{vor-2U} becomes more restrictive for growing $p\in[1,+\infty)$.


\begin{rem}\label{rem:Ctn}
The norm condition~\eqref{vor-2U} in the previous definition combined with the denseness of $D(A)\subset X$ implies that there exists an \emph{observability map} $\Ctn\in\sL\bigl(X,\LpntnY\bigr)$ satisfying $\|\Ctn\|\le M$ such that
\begin{equation}\label{def:Ct}
(\Ctn x)(s)=CT(s)x\quad\text{for all }x\in D(A),\ s\in[0,\tn].
\end{equation}
\end{rem}

\bigbreak
Finally, we consider the system \ref{csu} with (possibly nonzero) $p$-admissible control and observation operators $B$ and $C$. In order to proceed we need first
%
%
%
the following compatibility condition, cf. \cite[Sect.~II.A]{Hel:76}. For more information and several related conditions see \cite[Thm.~5.8]{Wei:94b} and \cite[Def.~5.1.1]{Sta:05}. Recall that $Z=D(C)$.

\begin{defn}\label{def:reg} The triple $(A,B,C)$ (or the system \ref{csu}) is called \emph{compatible} if
for some $\lambda\in\rho(A)$ we have
\begin{equation}\label{bild}
\rg\bigl(\RlAme B\bigr)\subset Z.
\end{equation}
\end{defn}

If the inclusion \eqref{bild} holds for some $\lambda\in\rho(A)$, then it holds for all $\lambda\in\rho(A)$ by the resolvent equation. Moreover, the closed graph theorem implies that the operator
\[
 C\RlAme B \in\sL(\U)\quad\text{for all }\lambda \in \rho(A).
\]

Consider now a compatible control system \ref{csu}. Since we are only interested in the generator property of $A+P$ for some perturbation $P$, we can assume without loss of generality that the growth bound $\gb(A)<0$ and hence
\[
0\in\rho(A).
\]
Note that for the initial value $x_0=0$, the input-output map of \ref{csu} which maps a control $u(\p)$  to the corresponding observation $y(\p)$ by \eqref{eq:vpf} is formally given by
\[
u(\p)\mapsto y(\p)=C\int_0^{\p} T_{-1}(\p-s) B u(s)\ds.
\]
However, the right hand side does in general not make sense for arbitrary $u\in\LpntnU$ since the integral might give values $\notin Z=D(C)$. However, if
\[u\in\WnzpntnU:=\bigl\{ u \in\WzpntnU: u(0)=u'(0)= 0\bigr\},\]
by integrating twice by parts and using \eqref{bild} we obtain
\begin{align*}
\int_0^{r}T_{-1}(r-s) B u(s)\, \ds
&=-A^{-1}_{-1}\biggl(Bu(r) +A_{-1}^{-1}Bu'(r)-\int_0^{r} T(r-s) A_{-1}^{-1} B u''(s)\ds\biggr)\\
&\in Z.
\end{align*}

At this point it is reasonable to ask that the input-output map is continuous. This gives rise to the following definition.

\begin{defn} The pair $(B,C)\in\sL(U,\Xme)\times\sL(Z,\Y)$ (or the system \ref{csu}) is called \emph{jointly $p$-admissible} for some $1\le p<+\infty$ if $B$ is a $p$-admissible control operator,  $C$ is a $p$-admissible observation operator and there exist $t_0>0$ and $M\ge0$ such that
\begin{equation}
\int_0^\tn\Bigl\|C\int_0^{r} T_{-1}(r-s) B u(s)\ds\Bigr\|_\Y^p\dr\le M\cdot\|u\|_p^p\qquad\text{for all }u\in\WnzpntnU.\label{vor-3}
\end{equation}
\end{defn}

\begin{rem}\label{rem:Ftn}
If \ref{csu} is jointly $p$-admissible, then there exists a bounded \emph{input-output map}
\begin{align}\label{sFt}
&\Ftn\in\sL\Bigl(\LpntnU\Bigr)\quad\text{such that}\notag\\
(&\Ftn u)(\p)=C\int_0^\p T_{-1}(\p-s) B u(s)\ds
\quad\text{for all }u\in\WnzpntnU.
\end{align}
\end{rem}



We need one more definition.

\begin{defn}\label{def:add-feedb} An operator $F\in\sL(\U)$ is called a $p$-\emph{admissible} feedback operator for some $1\le p<+\infty$ if there exists $\tn>0$ such that $Id-F\Ftn\in\sL\bigl(\LpntnY\bigr)$ is invertible.
\end{defn}

Note that $F=Id\in\sL(\U)$ is admissible if $\|\Ftn\|<1$.

\section{The Weiss--Staffans Perturbation Theorem}\label{sec:SW}

In this section we present the main results of this note which can be considered as purely operator theoretic versions of perturbation theorems for abstract linear systems due to Weiss \cite[Thms.~6.1 and 7.2 (1994)]{Wei:94a} in the Hilbert space case and Staffans \cite[Thms.~7.1.2 and 7.4.5 (2000)]{Sta:05} in the general case. In particular, we avoid the use of the notions of abstract linear systems and Lebesgue extensions which are not needed if one is only interested in generators.
For related results see also \cite{Had:05} and \cite[Thms.~4.2 and 4.3]{Sal:87}.

\begin{thm}\label{thm:WS}
Assume that $(A,B,C)$ is compatible, $(B,C)$ is jointly $p$-admissible and that $Id\in\sL(\U)$ is a $p$-admissible feedback operator for some $1\le p<+\infty$. This means that there exist $1\le p<+\infty$, $\tn>0$ and $M\ge0$ such that
\begin{alignat*}{2}
\text{(i)}\quad&\rg\bigl(\RlAme B\bigr)\subset Z&&\text{for some }\lambda\in\rho(A),\\
\text{(ii)}\quad&\int_0^{t_0} T_{-1}(t_0-s) B u(s)\ds \in X                    &&\text{for all }u\in\LpntnU,\\
\text{(iii)}\quad&\int_0^{t_0}\bigl\|CT(s)x\bigr\|_U^p\ds \leq M\cdot\|x\|_X^p&&\text{for all }x\in D(A),\\
\text{(iv)}\quad&\int_0^\tn\Bigl\|C\int_0^r T_{-1}(r-s) B u(s)\ds\Bigr\|_\Y^p\dr\le M\cdot\|u\|_p^p\qquad&&\text{for all }u\in\WnzpntnU,\\
\text{(v)}\quad&1\in\rho(\Ftn)\text{, where }\Ftn\in\sL\bigl(\LpntnU\bigr)\text{ is given by \eqref{sFt}}.&&
\end{alignat*}
Then
\begin{equation}\label{eq:def-A_BC}
\ABC:=(A_{-1}+BC)|_{X},\quad D(\ABC):=\bigl\{x\in Z:(A_{-1}+BC)x\in X\bigr\}
\end{equation}
generates a $C_0$-semigroup $\St$ on the Banach space $X$. Moreover,  the perturbed semigroup
$\St$ verifies the variation of parameters formula
\begin{equation}\label{eq-var-konst-form}
S(t)x=T(t)x+\int_0^t\Tme(t-s)\cdot BC\cdot S(s)x\ds\quad\text{for all } t\ge0\text{ and }x\in D(\ABC).
\end{equation}
\end{thm}

For the proof we extend the controllability-, observability- and input-output maps introduced in Remarks~\ref{rem:Btn}, \ref{rem:Ctn} and \ref{rem:Ftn} on $\RR_+$ as follows.

\begin{lem} Let $(A,B,C)$ be compatible and $(B,C)$ jointly $p$-admissible for some $1\le p<+\infty$. If $\gb(A)<0$, then there exist
\begin{enumerate}[(i)]
\item a strongly continuous, uniformly bounded family $\Btt\subset\sL\bigl(\LpniU,X\bigr)$,
\item a bounded operator $\Ci\in\sL\bigl(X,\LpniY\bigr)$, and
\item a bounded operator $\Fi\in\sL\bigl(\LpniU\bigr)$
\end{enumerate}
such that
\begin{align}
&\Bt u:=\int_0^{t} T_{-1}(t_0-s) B u(s)\ds &&\text{for all }u\in\LpniU,\\
&(\Ci x)(s)=CT(s)x&&\text{for all }x\in D(A),\, s\in[0,+\infty),\\
(&\Fi u)(\p)=C\int_0^\p T_{-1}(\p-s) B u(s)\ds&&\text{for all }u\in\WnzpniU.
\end{align}

\end{lem}

\begin{proof}
The assertion for $\Btt$ was proved in \cite[Cor.~3.16]{BE:13}.  The assertion for $\Ci$ was shown in \cite[Lem.~3.9]{BE:13}. Finally, the assertion for $\Fi$ follows from \cite[Rem.~3.23]{BE:13}
\end{proof}

For $\mu\ge0$ we indicate in the sequel the controllability-, observability- and input-output maps associated to the triple $(A-\mu,B,C)$ with the superscript ``$ ^\mu$'', e.g.,
\[
(\Fim u)(\p)=C\int_0^\p e^{-\mu(\p-s)}T_{-1}(\p-s) B u(s)\ds\qquad\text{for all }u\in\WnzpniU.
\]
Next we give a condition such that the invertibility of $I - \Ftn$ (see condition $(v)$ of Theorem \ref{thm:WS}) implies the one of $I - \Fim$.

\begin{lem}\label{lem:1-in-rho-Fi}
Let the assumptions of Theorem~\ref{thm:WS} be satisfied. If for $\mu\ge0$
\begin{equation}\label{eq:est-S-lambda}
\bigl\|T(\tn)+\sB_{\tn} (1-\Ftn)^{-1}\sC_{\tn}\bigr\|<e^{\mu\tn}
\end{equation}
holds, then $1\in\rho(\Fim)$.
\end{lem}

\begin{proof}
Inspired by \cite[(2.6)]{SW:04} and the proof of \cite[Prop.2.1]{Wei:89} we consider the surjective isometry%
\footnote{For a vector $v$ we denote by $v^T$ the transposed vector.}
\begin{align*}
J:\LpnntnU\to\prod_{k=1}^n \LpntnU,\quad
u\mapsto(u_1,\ldots,u_n)^T,
\end{align*}
where $u_k:[0,\tn]\to\U$, $u_k(s):=u\bigl((k-1)\tn+s\bigr)$ and $\|(u_1,\ldots,u_n)^T\|_p^p:=\sum_{k=1}^n\|u_k\|^p$.

\smallbreak
Then $\Fntn$ is isometrically isomorphic to the matrix
\[J\,\Fntn\Ji=
\left(
  \begin{array}{cccccc}
    \Ftn & 0 & 0 & \ldots & \ldots & 0 \\
    \sC_{\tn}{{T(\tn)^0}}\sB_{\tn} & \Ftn & 0 & \ddots &  & \vdots \\
    \sC_{\tn}T(\tn)^1\sB_{\tn} & \sC_{\tn}\sB_{\tn} &\ddots & \ddots& \ddots & \vdots \\
    \vdots & \ddots & \ddots & \ddots & 0 & 0 \\
    \vdots &   & \ddots & \sC_{\tn}\sB_{\tn} & \Ftn & 0 \\
    \sC_{\tn}T(\tn)^{n-2}\sB_{\tn} & \ldots & \ldots & \sC_{\tn}T(\tn)\sB_{\tn} & \sC_{\tn}\sB_{\tn} & \Ftn \\
  \end{array}
\right).
\]
Since by assumption $1-\Ftn$ is invertible, $1-\Fntn$ is invertible as well and $J(1-\Fntn)^{-1}\Ji=$
\[\hss\scriptscriptstyle
{
\left(
  \begin{array}{cccccc}
    \sG & 0 & 0 & \ldots & \ldots & 0 \\
     \sG\sC_{\tn}{{\bigl(T(\tn)+\sB_{\tn} \sG\sC_{\tn}\bigr)^0}}\sB_{\tn} \sG & \sG & 0 & \ddots &  & \vdots \\
    \sG\sC_{\tn}\bigl(T(\tn)+\sB_{\tn} \sG\sC_{\tn}\bigr){^1}\sB_{\tn} \sG & \sG\sC_{\tn}\sB_{\tn} \sG &\ddots & \ddots& \ddots & \vdots \\
    \vdots & \ddots & \ddots & \ddots & 0 & 0 \\
    \vdots &   & \ddots & \sG\sC_{\tn}\sB_{\tn} \sG & \sG & 0 \\
    \sG\sC_{\tn}\bigl(T(\tn)+\sB_{\tn} \sG\sC_{\tn}\bigr)^{n-2}\sB_{\tn} \sG & \ldots & \ldots & \sG\sC_{\tn}\bigl(T(\tn)+\sB_{\tn} \sG\sC_{\tn}\bigr)\sB_{\tn} \sG & \sG\sC_{\tn}\sB_{\tn} \sG & \sG \\
  \end{array}
\right)}
\hss\]
where we put $\sG:=(1-\Ftn)^{-1}$. By Lemma~\ref{lem:est-n-Toeplitz} applied to $J(1-\Fntn)^{-1}\Ji$ we obtain the estimate
\begin{equation}\label{eq:est-(1-F)inv}
\|(1-\Fntn)^{-1}\|\le
\|\sG\|+\|\sG\sC_{\tn}\|\cdot\|\sB_{\tn}\sG\|\cdot\sum_{l=1}^{n-1}\bigl\|\bigl(T(\tn)+\sB_{\tn} \sG\sC_{\tn}\bigr)\bigr\|^{l-1}.
\end{equation}
This shows that $\|(1-\Fntn)^{-1}\|$ remains bounded as $n\to+\infty$ if \eqref{eq:est-S-lambda} holds for $\mu=0$.

If the estimate  \eqref{eq:est-S-lambda} only holds for some $\mu>0$, we consider the triple $(A-\mu,B,C)$. Let $\Mem\in\sL\big(\LpntnU\big)$ be the multiplication operator defined by
\[
(\Mem u)(s):=e^{\mu s}\cdot u(s),\qquad u\in\LpntnU.
\]
Then $\Mem$ is invertible with inverse $\Mmem$ and a simple computation shows that
\begin{equation}\label{eq:sim-Ftn}
\Btn^\mu=e^{-\mu\tn}\Btn\Mem,
\qquad
\Ctn^\mu=\Mem^{-1}\Ctn\qquad\text{and}\qquad
\Ftnm=\Mem^{-1}\Ftn\Mem.
\end{equation}
By similarity this implies that $1\in\rho(\Ftnm)$. Hence we can repeat the above reasoning for $(A-\mu,B,C)$ and obtain from \eqref{eq:est-(1-F)inv} that $\|(1-\Fntnm)^{-1}\|$ remains bounded as $n\to+\infty$ if
\begin{equation}\label{eq:Sl<1}
\Bigl\|e^{-\mu\tn}T(\tn)+\Btn^\mu(1-\Ftnm)^{-1}\Ctn^\mu\Bigr\|<1.
\end{equation}
Since by \eqref{eq:sim-Ftn} we have
\begin{align*}
e^{-\mu\tn}T(\tn)+\Btn^\mu(1-\Ftnm)^{-1}\Ctn^\mu
&=e^{-\mu\tn}\bigl(T(\tn)+\Btn(1-\Ftn)^{-1}\Ctn\bigr),
\end{align*}
we conclude that the estimates \eqref{eq:Sl<1} and \eqref{eq:est-S-lambda} are equivalent.
Summing up this shows that \eqref{eq:est-S-lambda} implies that
\begin{equation}
K:=\sup_{n\in\NN}\l\|(1-\Fntnm)^{-1}\r\|<+\infty.
\end{equation}
Using this fact we finally show that $1\in\rho(\Fim)$.
Observe first that $(1-\Fim)u=0$ for some $u\in\LpniU$ implies that $(1-\Fntnm)(u|_{[0,n\tn]})=0$ for every $n\in\NN$. Since $(1-\Fntnm)$ is injective for every $n\in\NN$, this implies that $u=0$, i.e., $1-\Fim$ is injective.

\smallbreak
To show surjectivity we fix some $v\in\LpniU$ and define for $n\in\NN$
\[
u_n:=(1-\Fntnm)^{-1}(v|_{[0,n\tn]})\in\LpnntnU,
\]
i.e., $u_n$ is the unique solution in $\LpnntnU$ of the equation
\begin{equation}\label{eq:eq-dev-vn}
(1-\Fntnm)u=v|_{[0,n\tn]}.
\end{equation}
However, for $m\ge n$ we have $(\Fmtnm u_m)|_{[0,n\tn]}=\Fntnm (u_m|_{[0,n\tn]})$, hence
also $u_m|_{[0,n\tn]}\in\LpnntnU$ solves \eqref{eq:eq-dev-vn}. This implies that
\[
u_m|_{[0,n\tn]}=u_n.
\]
Thus we can define
\[
u(s):=\lim_{n\to+\infty}u_n(s),\qquad s\in[0,+\infty).
\]
Since $\|u_n\|\le K\cdot\|v\|$ for all $n\in\NN$, Fatou's lemma implies that $u\in\LpniU$. Moreover, by construction
\[
\bigl((1-\Fim)u\bigr)|_{[0,n\tn]}=(1-\Fntnm)u_n=v|_{[0,n\tn]}\qquad\text{for all }n\in\NN,
\]
which implies $((1-\Fim)u=v$. Since $v\in\LpntnU$ was arbitrary, this shows that $1-\Fim$ is surjective. Hence $1-\Fim$ is bijective and therefore $1\in\rho(\Fim)$ as claimed.
\end{proof}

Next we show that the invertibility of $Id-\Fim$ implies the invertibility of the \emph{``transfer function''} $Id-C\RlAme B$ of the system \ref{csu} with feedback $u(t)=y(t)$. Here we use the notations
\[
(\sL u)(\lambda):=\hat u(\lambda):=\int_0^{+\infty}e^{-\lambda r}u(r)\dr
\]
for the Laplace transform of a function $u$.

\begin{lem}\label{lem:H(lambda)}
Assume that $1\in\rho(\Fim)$ for some $\mu\ge0$. Then $1\in\rho\bigl(C\RlAme B\bigr)$ for all $\lambda\in\CC$ satisfying $\Re\lambda>\mu$ and
\begin{equation*}
\sL\bigl((Id-\Fim)^{-1}u\bigr)(\lambda)=\bigl(Id-C\RlAme B\bigr)^{-1}\cdot\hat u(\lambda)\quad\text{for all }u\in\LpniY.
\end{equation*}
\end{lem}

\begin{proof} Assume first that $\mu=0$. Then it is well known that $\Fi=\Fi^\mu$ is shift invariant (cf. \cite{Wei:91}), i.e. $\Fi$ commutes with the right shift. Then also $\sG:=Id-\Fi\in\sL\bigl(\LpniY\bigr)$ is shift invariant and by \cite[Thm.2.3]{Wei:91} and \cite[Lem.~3.19]{BE:13}  we obtain for $u\in\LpniY$
\[
\widehat{(\sG u)}(\lambda)=\bigl(Id-C\RlAme B\bigr)\cdot \hat u(\lambda),
\quad \Re\lambda>0.
\]
Let $\sR:=\sG^{-1}\in\sL\bigl(\LpniY\bigr)$. Then clearly the right shift also commutes with $\sR$, i.e. this operator is shift invariant as well. Hence again by  \cite[Thm.2.3]{Wei:91} there exists $R(\lambda)\in\sL(U)$ such that
\[
\widehat{(\sR u)}(\lambda)=R(\lambda)\cdot \hat u(\lambda),
\quad \Re\lambda>0,\; u\in\LpniY.
\]
Summing up we obtain for all $u\in\LpniY$
\begin{align*}
\hat u(\lambda)
&=\widehat{(\sR\sG u)}(\lambda)=R(\lambda)\cdot\widehat{(\sG u)}(\lambda)\\
&=R(\lambda)\cdot\bigl(Id-C\RlAme B\bigr)\cdot \hat u(\lambda)\\
&=\widehat{(\sG\sR u)}(\lambda)=\bigl(Id-C\RlAme B\bigr)\cdot\widehat{(\sR u)}(\lambda)\\
&=\bigl(Id-C\RlAme B\bigr)\cdot R(\lambda)\cdot \hat u(\lambda).
\end{align*}
If we take $u(s)=e^{-s}v$ for some $v\in\Y$, this implies
\begin{align*}
\tfrac{1}{1+\lambda}\cdot v
&=R(\lambda)\cdot\bigl(Id-C\RlAme B\bigr)\cdot\tfrac{1}{1+\lambda}\cdot v\\
&=\bigl(Id-C\RlAme B\bigr)\cdot R(\lambda)\cdot\tfrac{1}{1+\lambda}\cdot v,
\quad \Re\lambda >0.
\end{align*}
Hence $R(\lambda)=\bigl(Id-C\RlAme B\bigr)^{-1}$.

\smallbreak
If $\mu>0$, then by the same reasoning applied to $\Fim$ we obtain that
\[
1\in\rho\bigl(CR(\lambda,\Ame-\mu)B\bigr)=\rho\bigl(CR(\lambda+\mu,\Ame) B\bigr)
\qquad\text{for all }\Re\lambda>0.
\]
Clearly this implies our claim in case $\mu>0$ and the proof is complete.
\end{proof}

We are now well prepared to prove the main result of this section.

\begin{proof}[Proof of Theorem~\ref{thm:WS}]
The idea of the proof is to define an operator family $\St\subset\sLX$ and then to verify that it is a $C_0$-semigroup with generator $\ABC$.

To this end we first assume that the condition \eqref{eq:est-S-lambda} in Lemma~\ref{lem:1-in-rho-Fi} holds for $\mu=0$. Then $Id-\Fi$ is invertible, and for $t\ge0$ we can define
\begin{equation}
S(t):=T(t)+\Bt(Id-\Fi)^{-1}\Ci\in\sLX.
\end{equation}
Since $\Tt$ and $\Btt$ are both strongly continuous and uniformly bounded, the same holds for $\St$. We proceed and compute the Laplace transform of $S(\p)x:[0,+\infty)\to X$ for $x\in X$. Since
\begin{equation}\label{eq:def-S(.)}
S(\p)x=T(\p)x+T_{-1}(\p)B*(1-\Fi)^{-1}\Ci x,
\end{equation}
the convolution theorem for the Laplace transform (or \cite[Lem.~3.12]{BE:13}) and Lemma~\ref{lem:H(lambda)} imply for every $x\in X$ and $\Re\lambda>0$
\begin{align}
\sL\bigl(S(\p)x\bigr)(\lambda)
&=\RlA x+\RlAme B\cdot\sL\bigl((1-\Fi)^{-1}\Ci x\bigr)(\lambda)\notag\\
&=\RlA x+\RlAme B\cdot \bigl(Id-C\RlAme B\bigr)^{-1}\cdot C\RlA x\notag\\
&=:\Ql x.\label{eq:def-Q(lam)}
\end{align}
We will show that $Q(\lambda)=\RlABC$. First note that by the compatibility condition \eqref{bild} we have
\[
\rg\bigl(\Ql\bigr)\subset D(A)+Z=Z=D(C).
\]
Moreover,
\begin{align*}
(\lambda-\Ame-BC)&\cdot \Ql=\\
&=Id-BC\RlA + B\cdot Id\cdot\bigl(Id-C\RlAme B\bigr)^{-1} C\RlA\\
&\kern43pt                      -B\cdot C\RlAme B\cdot\bigl(Id-C\RlAme B\bigr)^{-1} C\RlA\\
&=Id.
\end{align*}
This implies that $\Ql$ is a right inverse and $\rg\bigl(\Ql\bigr)\subset D(\ABC)$. To show that it is also a left inverse we take $x\in D(\ABC)\subset Z=D(C)$. Then we obtain
\begin{align*}
\Ql&\cdot(\lambda-\Ame-BC) x=\\
&=x-\RlAme BCx
 +\RlAme B \bigl(Id-C\RlAme B\bigr)^{-1}\cdot Id\cdot Cx\\
&\kern56pt-\RlAme B \bigl(Id-C\RlAme B\bigr)^{-1}\cdot C\RlAme B\cdot Cx\\
&=x
\end{align*}
and hence it follows $Q(\lambda)=\RlABC$ as claimed.
Summing up we showed that $\St\subset\sLX$ is a strongly continuous family with Laplace transform $\RlABC$. By \cite[Thm.~3.1.7]{ABHN:01} this implies that $\St$ is a $C_0$-semigroup with generator $\ABC$.

\smallskip
To verify the variation of parameters formula \eqref{eq-var-konst-form} we first note that by Lemma~\ref{lem:H(lambda)} and the explicit representation of $\RlABC$ in \eqref{eq:def-Q(lam)} we have for all $x\in D(\ABC)$ and $\Re\lambda>\mu=0$ that
\[
\sL\bigl((1-\Fi)^{-1}\Ci(\p)x)\bigr)(\lambda)=\sL\bigl(C S(\p)x\bigr)(\lambda).
\]
By the uniqueness of the Laplace transform this implies that
\[
(1-\Fi)^{-1}\Ci(\p)x=C S(\p)x,
\]
and the assertion follows from the definition of $\St$ in \eqref{eq:def-S(.)}.

\smallbreak
Now assume that \eqref{eq:est-S-lambda} only holds for some $\mu>0$. Then we repeat the same reasoning for the triple $(A-\mu,B,C)$ and conclude as before that $(A-\mu)_{BC}=\bigl((A-\mu)_{-1}+BC\bigr)|_{X}=\ABC-\mu$ is a generator. Clearly this implies that $\ABC$ generates a strongly continuous semigroup $\St$.
Moreover we obtain that the pair of rescaled semigroups $\bigl(e^{-\mu t}T(t)\bigr)_{t\ge0}$ and $\bigl(e^{-\mu t}S(t)\bigr)_{t\ge0}$ verify the variation of parameters formula~\eqref{eq-var-konst-form} which implies that this formula holds for the pair $\Tt$ and $\St$ as well.
\end{proof}

As already remarked in the introduction, with increasing $p\in[1,+\infty)$ the $p$-admissibility of the control and observation operator becomes weaker and stronger, respectively.

If we assume that the input-output map maps $L^\a$ to $L^\b$ for some%
\footnote{The main cases we have in mind are $\a<p=\b<+\infty$ and $1<\a=p<\b$.}
$1\le \a\le p\le \b<+\infty$ satisfying $\a<\b$, we can drop the invertibility condition $1\in\rho(\Ftn)$ in Theorem~\ref{thm:WS} (and sometimes even the compatibility condition \eqref{bild}, cf. Remark~\ref{rem-comp}).

\begin{thm}\label{thm-main}
Assume that conditions~(i)--(iii) in Theorem~\ref{thm:WS} are satisfied. Moreover suppose
there exist $1\le \a\le p\le\b<\infty$ with $\a<\b$, $p>1$, and $M\ge0$ such that
\[\tag{\textit{iv'}}
\int_0^\tn\Bigl\|C\int_0^{r} T_{-1}(r-s) B u(s)\ds\Bigr\|_U^\b\dr\le M\cdot\|u\|_\a^{\b}\qquad\text{for all }u\in\WnzantnU.
\]
Then $\ABC$ given by \eqref{eq:def-A_BC}
generates a $C_0$-semigroup $\St$ on the Banach space $X$ which verifies the variation of parameters formula \eqref{eq-var-konst-form}.
\end{thm}

\begin{proof} By Theorem~\ref{thm:WS} it suffices to show that $1\in\rho(\Fte)$ for some $t_1>0$.
By assumption the operator
\begin{equation}\label{eq:Vtilde}
\widetilde\Vt:\WnzantU\subset\LantU \to\LbntU,\ \widetilde\Vt u:=C\int_0^\p\Tme(\p-s)Bu(s)\ds
\end{equation}
has a bounded extension $\Vt\in\sL\bigl(\LantU,\LbntU\bigr)$ for every $t\in(0,\tn]$. We distinguish 2 cases and use in both of them Jensen's inequality
\begin{equation}\label{eq:Jensen}
\|u\|_r\le t^{\frac1r-\frac1s}\cdot\|u\|_s
\end{equation}
for $1\le r<s\le+\infty$ and $u\in\LsntU\subset\LrntU$.

\emph{Case $\a<p$:} Then $\Vt$ belongs to $\sL\bigl(\LantU,\LpntU\bigr)$ with norm%
\footnote{We denote the norm of a bounded linear operator $F:\rL^r\to\rL^s$ by $\|F\|_{rs}$.}
$\|\Vt\|_{\a p}\le\|\Vtn\|_{\a p}$. This implies, by \eqref{eq:Jensen} for $r=\a$ and $s=p$, that
\[
\|\Vt u\|_p\le\|\Vt\|_{\a p}\cdot\|u\|_\a\le t^{\frac1\a-\frac1p}\cdot\|\Vtn\|_{\a p}\cdot\|u\|_p.
\]

\emph{Case $p<\b$:} In this case, $\Vt\in\sL\bigl(\LpntU,\LbntU\bigr)$ with norm
$\|\Vt\|_{p\b}\le\|\Vtn\|_{p\b}$. This implies, by \eqref{eq:Jensen} for $r=p$ and $s=\b$, that
\[
\|\Vt u\|_p\le t^{\frac1p-\frac1\b}\cdot\|\Vt u\|_\b\le t^{\frac1p-\frac1\b}\cdot\|\Vtn\|_{p\b}\cdot\|u\|_p.
\]

Hence in both cases, considering $\Ft:=\Vt|_{\LpntUs}\in\sL\bigl(\LpntU\bigr)$, we conclude that there exists $t_1>0$ such that $\|\Fte\|<1$ which implies $1\in\rho(\Fte)$.
\end{proof}

\begin{rem}\label{rem-comp}
As in Theorem~\ref{thm-main} assume that  $1\le\a\le p\le\b\le+\infty$ with $\a<\b$ and $1<p<+\infty$. If there exist $\tn>0$ and
a dense subspace $\D\subset\LantnU$ such that for every $u\in D$
\begin{itemize}
\item $\int_0^{r} T_{-1}(r-s) B u(s)\ds\in Z$ for almost all $0<r\le\tn$,
\item the map $[0,\tn]\ni r\mapsto C\int_0^{r} T_{-1}(r-s) B u(s)\ds$
is in $\LbntnU$, and
\item there exists $M\ge0$ such that
\[\int_0^\tn\Bigl\|C\int_0^{r} T_{-1}(r-s) B u(s)\ds\Bigr\|_U^\b\dr\le M\cdot\|u\|_\a^{\b}
\quad\text{for all }u\in\D,
\]
\end{itemize}
then also the compatibility condition \eqref{bild} is satisfied.

\smallbreak
To verify this assertion we define $\widetilde\Vt:D\subset\LantU \to\LbntU$ as in \eqref{eq:Vtilde} with $\WnzantU$ replaced by the space $D$. By assumption, $\widetilde\Vt$ has a unique bounded extension $\Vt:\LantU \to\LbntU$. As above take $\Ft:=\Vt|_{\LpntUs}\in\sL\bigl(\LpntU\bigr)$. Then, by H\"{o}lder's inequality (or \eqref{eq:Jensen} for $r=1$ and $s=\b$), we obtain for every $v\in U$\footnote{Here we define $(f\otimes x)(s):=f(s)\cdot x$ for all $s\in[0,\tn]$ where $f:[0,\tn]\to\CC$ is a scalar function. Moreover, $\eins$ denotes the constant one function on the interval $[0,\tn]$.}
\begin{align}
\biggl\|\frac1t\int_0^t(\Ft\, \eins\otimes v)(s)\ds\biggr\|
&\le\frac1t\int_0^t\bigl\|(\Ft\, \eins\otimes v)(s)\bigr\|\ds\notag\\
&=\frac1t\cdot\|\Vt\,\eins\otimes v\|_1\notag\\
&\le\frac1t\cdot t^{1-\frac1\b}\cdot\|\Vt\, \eins\otimes v\|_\b\notag\\
&\le t^{-\frac1\b}\cdot\|\Vt\|_{\a\b}\cdot\|\eins\otimes v\|_\a\notag\\
&\le t^{\frac1\a-\frac1\b}\cdot\|\Vtn\|_{\a\b}\cdot\|v\|_U\notag\\
&\to0\quad\text{as }t\to0^+.\label{eq:regular}
\end{align}
By \cite[Thm.~5.8]{Wei:94b} in the Hilbert space case or \cite[Thms.~5.6.4 \& 5.6.5]{Sta:05} in the general case this convergence implies the compatibility condition \eqref{bild}.
\end{rem}

%
%

%
%

\section{Applications}\label{sec:app}

We now give some applications of our abstract results. First we show that Theorem~\ref{thm-main} can be considered as a simultaneous generalization of the Desch--Schappacher and the Miyadera-Voigt perturbation theorems. Moreover, we generalize a result of Greiner concerning the perturbation of the boundary conditions of a generator.

\subsection{The Desch--Schappacher Perturbation Theorem}
The following result was proved in \cite[Thm.~5, Prop.8]{DS:89}, see also \cite[Cor.~III.3.4]{EN:00} and \cite[Cor.~5.5.1]{WT:09}. 

\begin{thm}[Desch--Schappacher, 1989]\label{thm:DS} Assume that for $B\in\sL(X,\Xme)$ there exist $1\le p<+\infty$, $t_0>0$ and $M\ge0$ such that
\begin{equation}\label{vor-1X}
\int_0^{t_0} T_{-1}(t_0-s) B u(s)\ds \in X\quad\text{for all }u\in\LpntnX.
\end{equation}
Then $(A_B,D(A_B))$ given by
\[
A_Bx:=(\Ame+B)x,\quad
D(A_B):=\bigl\{x\in X:(\Ame+B)x\in X\bigr\}
\]
is the generator of a $C_0$-semigroup $\St$ on $X$.
\end{thm}

We remark that one could consider the condition~\eqref{vor-1X} also for $p=\infty$ or $u\in\rC\bigl([0,\tn],U\bigr)$. However in this case one needs an additional norm estimate to ensure that condition (v) in Theorem~\ref{thm:WS} is satisfied, cf.  \cite[Cor.~III.3.3]{EN:00}. Moreover, we note that in a certain sense the Desch-Schappacher Theorem depends only on the range but not on the ``size'' of the perturbation $B$. In particular, if $B$ satisfies the assumption of Theorem~\ref{thm:DS}, then also $BF$ satisfies it for every $F\in\sL(X)$.


\begin{proof}[Proof of Theorem~\ref{thm:DS}]
Let $U=Z=X$ and $C=Id$. Then by assumption $B\in\sL(X,\Xme)$ is a $p$-admissible control operator and conditions~(i)--(iii) in Theorem~\ref{thm:WS} are clearly satisfied. We will prove that (ii) implies condition~(iv') from Theorem~\ref{thm-main}.
To this end we first verify that the function
\[
[0,\tn]\ni r\mapsto v(r):=\int_0^{r} T_{-1}(r-s) B u(s)\ds\in X
\]
is continuous for every $u\in\LpntnX$. For such $u$ define first $u_t:[0,\tn]\to U$ by
\begin{equation}\label{def-u_t}
u_t(s):=
\begin{cases}
0&\text{if }0\le s\le\tn-t\\
u(s-\tn+t)&\text{if }\tn-t< s\le\tn,
\end{cases}
\end{equation}
i.e., $u_t$ is just the right translation of $u$ by $\tn-t$.
Then $u_t\in\LpntnX$ and using Remark~\ref{rem:Btn} we obtain from $v(r)=\Btn u_r$ that for $r_0,\,r_1\in[0,\tn]$
\begin{equation*}
\|v(r_0)-v(r_1)\|=\|\Btn(u_{r_0}-u_{r_1})\|\le\|\Btn\|\cdot\|u_{r_0}-u_{r_1}\|_p\to0\quad\text{as }r_1\to r_0,
\end{equation*}
where the last step follows from the strong continuity of the nilpotent right translation semigroup on $\LpntnX$. Next we define the operator
\[
\Vtn:\LpntnX\to\CntnX,
\quad
(\Vtn u)(r):=\int_0^{r} T_{-1}(r-s) B u(s)\ds,\ r\in[0,\tn].
\]
By what we just showed, $\Vtn$ is well-defined. Moreover, the estimate
\[
\|(\Vtn u)(r)\|\le\|\Btn\|\cdot\|u_r\|_p\le\|\Btn\|\cdot\|u\|_p\quad\text{for all }u\in\LpntnX,\ r\in[0,\tn]
\]
shows that $\Vtn\in\sL\bigl(\LpntnX,\CntnX\bigr)\subset\sL\bigl(\LpntnX,\LbntnX\bigr)$ for all $\beta\ge1$. Choosing $\beta>p$ this implies condition~(iv') and hence the proof is complete.
\end{proof}

\begin{rem}
The proofs of Theorems~\ref{thm-main} and \ref{thm:DS} imply the following: If $B\in\sL(U,\Xme)$ is a $p$-admissible control operator then for every bounded $C\in\sL(X,U)$ the triple $(A,B,C)$ is compatible and jointly $p$-admissible. Moreover, in this case every $F\in\sL(U)$ is a $p$-admissible feedback operator for the system \ref{csu}.
\end{rem}

\subsection{The Miyadera--Voigt Perturbation Theorem}

As another application we consider the following version of the Miyadera--Voigt perturbation theorem, cf. \cite{Miy:66} and \cite{Voi:77}, see also \cite[Cor.~III.3.16]{EN:00} and \cite[Thm.~5.4.2]{WT:09}.

\begin{thm}[Miyadera 1966/Voigt 1977]\label{thm:MV} Assume that for $C\in\sL(X_1^A,X)$ there exist $1<p<+\infty$, $t_0>0$ and $M\ge0$ such that
\begin{equation}\label{vor-2X}
\int_0^{t_0}\bigl\|CT(s)x\bigr\|_X^p\ds \leq M\cdot\|x\|_X^p
\quad\text{for all }x\in D(A).
\end{equation}
Then $(A_C,D(A_C))$ given by
\[
A_Cx:=(A+C)x,\quad
D(A_C):=D(A),
\]
is the generator of a $C_0$-semigroup on $X$.
\end{thm}

We remark that one could consider condition~\eqref{vor-2X} also for $p=1$. However in this case one needs $M<1$ to ensure that condition (v) in Theorem~\ref{thm:WS} is satisfied, cf. \cite[Cor.~III.3.16]{EN:00}. Moreover, we note that in a certain sense the Miyadera-Voigt Theorem~\ref{thm:MV} (for $p>1$) depends only on the domain but not on the ``size'' of the perturbation $C$. In particular, if $C$ satisfies the assumption of Theorem~\ref{thm:MV}, then also $FC$ satisfies it for every $F\in\sL(X)$.

\begin{proof}[Proof of Theorem~\ref{thm:MV}]
Let $U=X$, $Z=X_1^A$ and $B=Id$. Then, by assumption, $C\in\sL(Z,X)$ is a $p$-admissible observation operator and conditions~(i)--(iii) in Theorem~\ref{thm:WS} are clearly satisfied. We will show that condition~ (iii) implies condition~(iv') from Theorem~\ref{thm-main}. To this end fix $0\le\aaa<\bb\le\tn$ and $x\in D(A)$. Then for $u=\eab\otimes x$ we obtain
\[
C\int_0^{r} T(r-s) u(s)\ds
=CA^{-1}\int_\aaa^r\eab(s)\cdot T(r-s)Ax\ds
=\int_\aaa^r\eab(s)\cdot CT(r-s)x\ds.
\]
Using this, condition~(iii), the triangle- and  H\"{o}lder's inequality
\[
\l(\int_a^b |f(s)|\ds\r)^p\le(b-a)^{p-1}\int_a^b |f(s)|^p\ds
\]
for $f\in\rL^1(a,b)$ and $p\ge1$ we obtain
\begin{align}
\int_0^\tn\Bigl\|&C\int_0^{r} T(r-s)u(s)\ds\Bigr\|_X^p\dr
\le\int_\aaa^\tn\Bigl(\int_\aaa^{r} \eab(s)\cdot\bigl\|CT(r-s)x\bigr\|_X\ds\Bigr)^p\dr\notag\\
&=\int_\aaa^\bb\Bigl(\int_\aaa^{r} \bigl\|CT(r-s)x\bigr\|_X\ds\Bigr)^p\dr
 +\int_\bb^\tn\Bigl(\int_\aaa^{\bb} \bigl\|CT(r-s)x\bigr\|_X\ds\Bigr)^p\dr\notag\\
&\le\int_\aaa^\bb(r-\aaa)^{p-1}\int_\aaa^{r}\bigl\|CT(r-s)x\bigr\|_X^p\ds\dr
 +\int_\bb^\tn\!\!\!\int_\aaa^{\bb}(\bb-\aaa)^{p-1}\bigl\|CT(r-s)x\bigr\|_X^p\ds\dr\notag\\
&\le\int_\aaa^\bb(r-\aaa)^{p-1}M\cdot\|x\|^p\dr
 +\int_\aaa^{\bb}(\bb-\aaa)^{p-1}\int_\bb^\tn\bigl\|CT(r-s)x\bigr\|_X^p\dr\ds\notag\\
&\le \tfrac{M}{p}\cdot(\bb-\aaa)^p\cdot\|x\|^p
 +\int_\aaa^{\bb}(\bb-\aaa)^{p-1}M\cdot\|x\|^p\ds\notag\\
&=  M\bigl(1+\tfrac{1}{p}\bigr)\cdot(\bb-\aaa)^p\cdot\|x\|^p
=:K^p\cdot(\bb-\aaa)^p\cdot\|x\|^p.\label{eq:est-Vt-Cadm}
\end{align}
Let now $u=\sum_{k=1}^n\eabk\otimes x_k\in\LentnX$ be a step function where the intervals $[\aaa_k,\bb_k]\subset[0,\tn]$ are pairwise disjoint and $x_k\in D(A)$ for $k=1\ldots n$. Then from \eqref{eq:est-Vt-Cadm} we obtain
\begin{align*}
\l(\int_0^\tn\Bigl\|C\int_0^{r} T(r-s)u(s)\ds\Bigr\|_X^p\dr\r)^{\frac1p}
&\le K\cdot\sum_{k=1}^n (\bb_k-\aaa_k)\cdot\|x_k\|_X\\
&=K\cdot\|u\|_1.
\end{align*}
Since the step functions having values in $D(A)$ are dense in $\LentnX$, this implies condition~(iv') for $\alpha=1$ and $\beta=p$. This completes the proof.
\end{proof}


\begin{rem}
The proofs of Theorems~\ref{thm-main} and \ref{thm:MV} imply the following: If $C\in\sL(Z,U)$ is a $p$-admissible observation operator then for every bounded $B\in\sL(U,X)$ the triple $(A,B,C)$ is compatible and jointly $p$-admissible. Moreover, in this case every $F\in\sL(U)$ is a $p$-admissible feedback operator for the system \ref{csu}.
\end{rem}

\subsection{Perturbing the Boundary Conditions of a Generator}

In this section we show how Theorem~\ref{thm:WS} can be used to generalize results by Greiner in \cite{Gre:87} on the perturbation of boundary conditions of a generator. 
\goodbreak
To explain the general setup we consider
\begin{itemize}
\item two Banach spaces\footnote{In this section we denote the elements of $X$ by $f$ instead of $x$.}
 $X$ and $\partial X$, the latter called ``boundary space'';
\item a closed, densely defined ``maximal'' operator\footnote{``maximal'' concerns the size of the domain, e.g., a differential operator without boundary conditions.} $A_m:D(A_m)\subseteq X\to
X$;
\item the Banach space $[D(A_m)]:=(D(A_m),\|\cdot\|_{A_m})$ where $\|x\|_{A_m}:=\|x\|+\|A_mx\|$ is the graph norm;
\item two ``boundary'' operators $L,\Phi\in\sL([D(A_m)],\partial X)$.
\end{itemize}

Then we define two restrictions $A,\,A^\Phi\subset A_m$ by
\begin{align*}
D(A):&=\bigl\{f\in D(A_m):Lf=0\bigr\}=\ker L,\\
D(A^\Phi):&=\bigl\{f\in D(A_m):Lf=\Phi f\bigr\}.
\end{align*}
In many applications $X$, $\partial X$ and $D(A_m)$ are function spaces and $L$ is a trace-type operator which restricts a function in $D(A_m)$ to (a part of) the boundary of its domain. Hence we can consider $A^\Phi$ with boundary condition $Lf=\Phi f$ as a perturbation of the operator $A$ with abstract boundary condition $Lf=0$.

In order to treat this setup within our framework  we make the following assumptions.

\begin{enumerate}[(i)]
\item The operator $A$ generates a strongly continuous semigroup $\Tt$ on $X$;
\item the boundary operator $L:D(A_m)\to\partial X$ is surjective.
\end{enumerate}

Under these assumptions the following lemma, shown by
Greiner \cite[Lem.~1.2]{Gre:87}, is the key to write $A^\Phi$ as a Staffans--Weiss type perturbation of $A$.

\begin{lem}\label{lem-Gre} Let the above assumptions (i) and (ii) be
satisfied. Then for each $\lambda\in\rho(A)$ the operator $L|_{\ker(\lambda-A_m)}$ is invertible and
$D_\lambda:=(L|_{\ker(\lambda-A_m)})^{-1}:\partial X\to\ker(\lambda-A_m)\subseteq X$
is bounded.
\end{lem}

Using this so-called \emph{Dirichlet operator} $D_\lambda$ we obtain the following representation of $A^\Phi$ where for simplicity we assume that $A$ is invertible.

\begin{lem}
If $0\in\rho(A)$, then
\begin{equation}\label{eq:def-A-Phi}
A^\Phi=\bigl(\Ame-\Ame D_0\cdot\Phi\bigr)|_X,
\end{equation}
i.e., $A^\Phi=\ABC$ for $U:=\partial X$, $Z:=[D(A_m)]$ and
\[
B:=-\Ame D_0\in\sL(U,\Xme),\qquad C:=\Phi\in\sL(Z,U).
\]
\end{lem}

\begin{proof}
Denote the operator on the right-hand side of \eqref{eq:def-A-Phi} by $\tilde A^\Phi$. Then
\begin{align*}
f\in D\bigl(\tilde A^\Phi\bigr)
&\iff f-D_0\Phi f\in D(A)\\
&\iff Lf=LD_0\Phi f=\Phi f\\
&\iff f\in D\bigl(A^\Phi\bigr).
\end{align*}
Moreover, for $f\in D(A^\Phi)$ we have
\[
\tilde A^\Phi f=A( f-D_0\Phi f)=A_m( f-D_0\Phi f)=A_mf=A^\Phi f
\]
as claimed.
\end{proof}

We mention that in \cite[Thm.~2.1]{Gre:87} the operator $\Phi\in\sL(X,U)$ is bounded and the assumptions imply that  $\Ame D_0$ is a $1$-admissible control operator. Hence in this case $A^\Phi$ is a generator by the Desch--Schappacher theorem.

\smallbreak
By using Theorem~\ref{thm:WS} we can now deal with unbounded $\Phi$.

\begin{cor}\label{cor:GG-ubdd}  Assume that  for some $1\le p<+\infty$ the pair $(\Ame D_0,\Phi)$ is jointly $p$-admissible and that $Id\in\sL(\partial X)$ is a $p$-admissible feedback operator for $A$. Then $A^\Phi$ is the generator of a $C_0$-semigroup on $X$.
\end{cor}

\begin{proof} We only have to show the compatibility condition~\eqref{bild}. This, however, immediately follows from
\[
\rg\bigl(\RlAme B\bigr)=
\rg\bigl((Id-\lambda\RlA)D_0\bigr)\subset
\ker(A_m)+D(A)\subseteq D(A_m)=
Z.\qedhere
\]
\end{proof}

\begin{rem} We note that in \cite[Thm.~4.1]{HMR:14} the authors study a similar problem in the context of regular linear systems.
\end{rem}

As a simple but typical example for the previous corollary we consider the space $X:=\Lpne$ and the first derivative $A_m:=\frac{d}{ds}$ with domain $D(A_m):=\Wepne$ (c.f. \cite[Expl.~1.1.(c)]{Gre:87}). As boundary space we choose $\partial X=\CC$, as boundary operators the point evaluation $L=\delta_1$ and as perturbation some $\Phi \in\bigl(\Wepne\bigr)'$. This gives rise to the differential operators $A,\,A^\Phi\subset\frac{d}{ds}$ with domains
\begin{align*}
D(A):&=\bigl\{f\in\Wepne:f(1)=0\bigr\},\\
D(A^\Phi):&=\bigl\{f\in\Wepne:f(1)=\Phi f\bigr\}.
\end{align*}
Then clearly the assumptions~(i) and (ii) made above are satisfied, in particular $A$ generates the nilpotent left-shift semigroup given by
\[
\bigl(T(t)f\bigr)(s)=
\begin{cases}
f(s+t)&\text{if }s+t\le1,\\
0&\text{else}.
\end{cases}
\]
However, $A^\Phi$ is not always a generator. For example if $\Phi=\delta_1$, then $A^\Phi=A_m$ and $\sigma(A^\Phi)=\CC$, hence $A^\Phi$ cannot be a generator. Thus we need an additional assumption on $\Phi$.

\begin{defn} A bounded linear functional $\Phi:\Cne\to\CC$ has little mass in $r=1$ if there exist $q<1$ and $\delta>0$ such that
\[
|\Phi f|\le q\cdot\|f\|_\infty
\]
for every $f\in\Cne$ satisfying $\supp f\subset[1-\delta,1]$.
\end{defn}

Note that $\Wepne\inc\Cne$ and hence $(\Cne)'\subset[D(A_m)]'$. Now the following holds.

\begin{cor}\label{cor:gen-nmiz}
If $\Phi\in\bigl(\Cne\bigr)'$ has little mass in $r=1$, then for all $1\le p<+\infty$ the operator $A^\Phi$ is the generator of a strongly continuous semigroup on $\Lpne$.
\end{cor}
\begin{proof}
By Corollary~\ref{cor:GG-ubdd} it suffices to show that the conditions (ii)--(v) of Theorem~\ref{thm:WS} are satisfied. To this end we first note that $0\in\rho(A)$ and the Dirichlet operator $D_0:\CC\to\Lpne$ is given by $D_0\alpha=\alpha\cdot\eins$ where $\eins(s)=1$ for all $s\in[0,1]$.


\smallbreak
(ii) By Remark~\ref{rem:Btn} it suffices to verify estimate \eqref{eq:add-B-M} where we may assume
 $u\in\Wepnntn$ for some $0<\tn\le1$. Using integration by parts and  \cite[Thm.~4.2]{Nei:81} we conclude%
\footnote{For a function $g$ defined on an interval we denote in the sequel by $\tilde g$ its extension to $\RR$ by the value $0$.}
\begin{align}\label{eq:Ft-expl}
\int_0^\tn\Tme(\tn-s)Bu(s)\ds
&=-\int_0^\tn\Tme(\tn-s)\Ame D_0u(s)\ds\notag\\
&=D_0u(\tn)-\int_0^\tn T(\tn-s)D_0u'(s)\ds\notag\\
&=u(\tn)\cdot\eins-\int_0^\tn\bigl(T(\tn-s)\eins\bigr)\cdot u'(s)\ds\notag\\
&=u(\tn)\cdot\eins-\int_{\max\{0,\ps+\tn-1\}}^\tn\kern-6ptu'(s)\ds\notag\\
&=u\bigl(\max\{0,\p+\tn-1\}\bigr)\notag\\
&=\tilde u(\p+\tn-1).
\end{align}
This implies $\|\Btn u\|_X=\|\Btn u\|_p\le\|u\|_p$  for all  $u\in\Wepnntn$ which shows (ii).
\smallbreak
(iii) By the Riesz--Markov representation theorem there exists a regular complex Borel measure $\mu$ on $[0,1]$ such that
\begin{equation}\label{rep-Phi}
\Phi f=\int_0^1 f(r)\dmr\quad\text{for all }f\in\Cne.
\end{equation}
Using Fubini's theorem and H\"{o}lder's inequality we obtain for $0<\tn\le1$ and $f\in D(A)$
\begin{align}\label{eq:est-C-admiss}
\int_0^\tn\bigl|CT(s)f\bigr|^p\ds
&=\int_0^\tn\bigl|\Phi\tilde f(\p+s)\bigr|^p\ds\notag\\
&\le\int_0^\tn\Bigl(\int_0^1 \bigl|\tilde f(r+s)\bigr|\dmar\Bigr)^p\ds\notag\\
&\le\int_0^\tn \bigl(|\mu|[0,1]\bigr)^{p-1}\cdot\int_0^1 \bigl|\tilde f(r+s)\bigr|^p\dmar\ds\notag\\
&=\|\mu\|^{p-1}\cdot\int_0^1\int_0^\tn \bigl|\tilde f(r+s)\bigr|^p\ds\dmar\notag\\
&\le\|\mu\|^{p}\cdot\|f\|_p^p,
\end{align}
where we put $\|\mu\|:=|\mu|[0,1]$ (which coincides with $\|\Phi\|_\infty$). This proves (iii).
\smallbreak
(iv) From \eqref{eq:Ft-expl} we obtain for $0<\tn\le1$ and  $u\in\Wepnntn$ by similar arguments as in (iii) 
\begin{align}\label{eq:est-Ft}
\int_0^\tn\Bigl|C\int_0^r\Tme(r-s)Bu(s)\ds\Bigr|^p\dr
&=\int_0^\tn\bigl|\Phi\,\tilde u(\p+r-1)\bigr|^p\dr\notag\\
&=\int_0^\tn\Bigl|\int_{1-r}^1u(s+r-1)\dms\Bigr|^p\dr\notag\\
&\le\int_0^\tn\bigl(|\mu|[1-r,1]\bigr)^{p-1}\cdot\int_{1-r}^1\bigl|u(s+r-1)\bigr|^p\dmas\dr\notag\\
&\le\bigl(|\mu|[1-\tn,1]\bigr)^{p-1}\cdot\int_{1-\tn}^1\int_{1-s}^1\bigl|u(s+r-1)\bigr|^p\dr\dmas\notag\\
&\le\bigl(|\mu|[1-\tn,1]\bigr)^{p}\cdot\|u\|_p^p.
\end{align}
This shows (iv).

\smallbreak
(v) Since by assumption $\Phi$ has little mass in $r=1$, it follows that $|\mu|[1-\tn,1]<1$ for sufficiently small $\tn>0$. Hence from estimate \eqref{eq:est-Ft} and the denseness of $\Wepnntn$ in $\Lpntn$ it follows that $\|\Ftn\|\le|\mu|[1-\tn,1]<1$ for $0<\tn\le 1$ sufficiently small. This implies $1\in\rho(\Ftn)$ as claimed.
\end{proof}

\begin{rems}
(i) Corollary~\ref{cor:gen-nmiz} could be easily generalized (with essentially the same proof) to the first derivative on $\rL^p\bigl([0,1],\CC^n\bigr)$. One could even go further and prove a similar result on $\rL^p\bigl([0,1],E\bigr)$ for a (possibly infinite dimensional) Banach space $E$ provided the boundary operator $\Phi$ has a representation as a Riemann-Stieltjes integral replacing \eqref{rep-Phi}.

\smallbreak
(ii) In most cases the admissibility of the identity as a feedback operator is verified by showing that $\|\Ftn\|<1$ for sufficiently small $\tn>0$. If we choose $\Phi=\alpha\delta_1$, by \eqref{eq:Ft-expl} we obtain $\Ftn=\alpha Id$ for all $\tn>0$, hence $1\in\rho(\Ftn)$ if and only if $\alpha\ne1$. This provides an example where our perturbation theorem is applicable even if $\|\Ftn\|>1$ for all $\tn>0$. Note that for $\alpha=1$ we obtain $A^\Phi=A_m$, hence in this case $A^\Phi$ cannot be a generator.
\end{rems}

\appendix

\section{Estimating the $p$-Norm of a Triangular Toeplitz Matrix}\label{sec:append}

For the proof of Lemma~\ref{lem:1-in-rho-Fi} we needed the following result.

\begin{lem}\label{lem:est-n-Toeplitz}
For a Banach space $X$ endow $\sX:=X^n$, $n\in\NN$, with the $p$-norm
\[\bigl\|(x_1,\ldots,x_n)^T\bigr\|_p:=\Bigl(\sum_{k=1}^n\|x_k\|^p\Bigr)^{\frac1p}\]
\end{lem}
for some $1\le p\le+\infty$. Moreover, let $T_0,\ldots,T_{n-1}\in\sL(X)$. Then the norm of the Toeplitz operator matrix
\[
\sT:=\bigl(T_{j-i}\bigr)_{i,j=1}^n=
\left(
  \begin{array}{cccccc}
    T_0   & 0  & 0  & \ldots & \ldots & 0 \\
    T_1 & T_0 & 0 & \ddots &  & \vdots \\
    T_2 & T_1 & \ddots & \ddots & \ddots & \vdots \\
    \vdots & \ddots & \ddots & \ddots & 0 & 0 \\
    \vdots &  & \ddots & T_1 & T_0 & 0 \\
    T_{n-1} & \ldots & \ldots & T_2 & T_1 & T_0 \\
  \end{array}
\right)_{n\times n}\kern-12pt\in\sL(\sX)
\]
can be estimated as
\[
\|\sT\|\le\sum_{j=0}^{n-1}\|T_j\|.
\]

\begin{proof}
Let $\sx=(x_1,\ldots,x_n)^T\in\sX$. 
Then we can estimate
\begin{align*}
\|\sT\sx\|_p
&=\Bigl(\sum_{j=1}^n\Bigl\|\sum_{i=1}^jT_{j-i}x_i\Bigr\|^p\Bigr)^{\frac1p}\\
&\le\Bigl(\sum_{j=1}^n\Bigl(\sum_{i=1}^j\|T_{j-i}\|\cdot\|x_i\|\Bigr)^p\Bigr)^{\frac1p}\\
&=\Bigl(\sum_{j=1}^n\Bigl(\bigl((\|T_0\|,\|T_1\|,\ldots,\|T_{n-1}\|)*(\|x_1\|,\|x_2\|,\ldots,\|x_n\|)\bigr)(j)\Bigr)^p\Bigr)^{\frac1p}\\
&=\Bigl\|\bigl(\|T_0\|,\|T_1\|,\ldots,\|T_{n-1}\|\bigl)*\bigr(\|x_1\|,\|x_2\|,\ldots,\|x_n\|\bigr)\Bigr\|_p\\
&\le\Bigl\|\bigl(\|T_0\|,\|T_1\|,\ldots,\|T_{n-1}\|\bigr)\Bigl\|_1\cdot\Bigr\|\bigl(\|x_1\|,\|x_2\|,\ldots,\|x_n\|\bigr)\Bigr\|_p\\
&=\sum_{j=0}^{n-1}\|T_j\|\cdot \|\sx\|_p
\end{align*}
where in the second last step we applied Young's inequality to the convolution of sequences.
\end{proof}

%
%

\end{document}